\documentclass[12pt,a4paper, oneside]{amsart}

\usepackage[top=85pt,bottom=85pt,left=65pt, right=65pt]{geometry}

\usepackage{amsmath, amsthm, amsfonts, amssymb, nccmath}
\usepackage[utf8]{inputenc}
\usepackage[foot]{amsaddr}
\usepackage{bbm}
\usepackage{mathtools}
\usepackage{cite} 
\usepackage{graphicx}
\usepackage{fnpct} 
\usepackage{enumitem, hyperref}
\usepackage{verbatim}

\hypersetup{colorlinks=true, linkcolor=red, urlcolor=blue, citecolor=blue, pdfpagemode=UseNone, pdfstartview=FitH, bookmarksopen=true}
\hypersetup{pdftitle=Title}

\usepackage[table,pdftex,dvipsnames]{xcolor}
\usepackage{xargs}                      
\usepackage{tikz}
\usepackage[matrix,arrow,cmtip,rotate,curve,arc]{xy}

\usetikzlibrary{decorations,decorations.pathreplacing, math}

\DeclareMathOperator\sd{sd}

\DeclareMathOperator{\conv}{conv}

\newcommand\eps{\ensuremath{\varepsilon}}
\newcommand{\pth}[1]{\left( #1 \right)}

\newcommand{\nontrivial}{nontrivial}
\newcommand{\simplex}[1]{\Delta_{#1}}
\newcommand{\skel}[2]{#2^{(#1)}}
\newcommand{\skelsim}[2]{\skel{#1}{\simplex{#2}}}

\newcommand\R{\ensuremath{\mathbb{R}}}
\newcommand\N{\ensuremath{\mathbb{N}}}
\newcommand\Z{\ensuremath{\mathbb{Z}}}
\newcommand\Q{\ensuremath{\mathbb{Q}}}

\newcommand{\Rspace}{\mathbf{R}}
\newcommand\A{\mathcal{A}}
\newcommand\F{\mathcal{F}}
\newcommand\G{\mathcal{G}}
\newcommand\C{\mathcal{C}}

\newcommand{\heading}[1]{\vspace{1ex}\par\noindent{\bf\boldmath #1}}

\begingroup
    \makeatletter
    \@for\theoremstyle:=definition,remark,plain\do{%
        \expandafter\g@addto@macro\csname th@\theoremstyle\endcsname{%
            \addtolength\thm@preskip\parskip
            }%
        }
\endgroup

\newtheorem{theorem}{Theorem}[section]

\newtheorem{lemma}[theorem]{Lemma}
\newtheorem{proposition}[theorem]{Proposition}

\theoremstyle{definition}

\newtheorem{problem}{Problem}
\newtheorem{definition}[theorem]{Definition}

\newtheorem{remark}[theorem]{Remark}
\newtheorem{example}{Example}

\begin{document}

\title{Bounding Radon numbers via Betti numbers}
\author{Zuzana Pat\'akov\'a}

\date{}
\address{\noindent Department of Algebra, Faculty of Mathematics and Physics, Charles University, Sokolovská~83, 186 75 Praha, Czech Republic}
\email{patakova@karlin.mff.cuni.cz}


\begin{abstract}
We prove general topological Radon-type theorems for sets in $\mathbb R^d$ or on a surface. Combined with a recent result of Holmsen and Lee, we also obtain fractional Helly theorem, and consequently the existence of weak $\varepsilon$-nets as well as a $(p,q)$-theorem for those sets.

More precisely, 
given a family $\F$ of subsets of $\R^d$, we will measure the homological complexity of $\F$ by the supremum of the first $\lceil d/2\rceil$ reduced Betti numbers of $\bigcap \G$ over all nonempty $\G \subseteq \F$. We show that if  $\F$ has homological complexity at most $b$, the Radon number of $\F$ is bounded in terms of $b$ and $d$. 
In case that $\F$ lives on a surface and the number of connected components  of $\bigcap \mathcal G$ is at most $b$ for any nonempty $\G \subseteq \F$, then the Radon number of $\F$ is bounded  by a function depending only on $b$ and the surface itself.
 
 For surfaces, if we moreover assume the sets in $\F$ are open, we show  that the fractional Helly number of $\mathcal F$ is linear in $b$. The improvement is based on a recent result of the author and Kalai. Specifically, for $b=1$ we get that the fractional Helly number is at most three, which is optimal. This case further leads to solving a conjecture of Holmsen, Kim, and Lee about an existence of a $(p,q)$-theorem for open subsets of a surface. 
\end{abstract}

\maketitle

\section{Introduction}
Radon's theorem \cite{Radon1921} is a central result in convex geometry. It states
that it is possible to split any $d+2$ points in $\R^d$ into two disjoint parts whose convex hulls intersect. 
It is natural to ask what happens to the statement, if one replaces convexity with a more general notion.

Perhaps the most versatile generalization of the convex hull is the following. Let $X$ be an underlying set and let $\mathcal F$ be a family of subsets of $X$. Let $S\subseteq X$ be a set. The \emph{$\F$-hull $\operatorname{conv}_{\mathcal F}(S)$ of $S$ relative to $\mathcal F$} is defined as the intersection of all sets 
from $\mathcal F$ that contain $S$. If there is no such set, the $\F$-hull is, by definition, $X$.  This definition is closely related to so called  \emph{convexity spaces}, which we discuss in Section~\ref{ss:convexvsclosed}. Note that if $\F$ is a family of all convex sets in $\R^d$, then $\conv_\F$ coincides with the standard convex hull.

The \emph{Radon number} $r(\mathcal F)$ of  $\mathcal F$  is the smallest integer $r$ such that any set $S\subseteq X$ of cardinality $r$ can be split into two disjoint parts $S=P_1\sqcup P_2$ satisfying $\operatorname{conv}_{\mathcal F}(P_1)\cap \operatorname{conv}_{\mathcal F}(P_2)\neq \emptyset$.
If no such $r$ exists, we put $r(\F) = \infty$.

 Radon numbers and their relatives are studied not only in convex or discrete geometry \cite[etc.]{eckhoff,tverberg1, BB, eckhoff_polytopes, KK}, but they found applications also in computational geometry \cite{clarkson}, computational complexity \cite{Onn, Miller2010}, optimization \cite{Brosowski,amenta2}, robust statistics \cite{amenta}, operation research \cite{QueyranneTardella}, or parallel machine learning \cite{machine-learning}.
\medskip

In this paper we show that very mild topological conditions are enough to force a bound on Radon number for sets in Euclidean space (Theorem \ref{t:boundedRadon}) or on surfaces (Theorem~\ref{t:surf}). A simple trick allows us to give a version of the result for smooth manifolds or simplicial complexes, see Section \ref{s:embeddability}. 
In Section~\ref{s:direct_conseq} we list some important consequences, most notably fractional Helly theorems (Theorems~\ref{t:non-optimal-frachelly} and \ref{t:fract_helly}), which, with some further tools (Theorem~\ref{t:frachelly_surfaces}) lead to solving a conjecture of Holmsen, Kim, and Lee \cite[Conj.~5.3]{nerves_minors} (a special case of Theorem~\ref{t:pq_surface}). In Section~\ref{ss:convexvsclosed} we discuss a relation between our setting and convexity spaces.

\section{Results}
Radon's theorem \cite{Radon1921}  says that the family of convex sets in $\mathbb R^d$ has Radon number $d+1$, and this remains true after applying an arbitrary topological deformation to $\mathbb R^d$. 
 Hence, bounded Radon number should not be seen as a property that holds merely for the family of standard convex sets. 
In this paper we prove that the Radon number $r(\F)$ of a family $\F$ is bounded whenever $\F$ is \emph{``not too topologically complicated''}.
Let us first explain what ``not too topologically complicated'' means.

\heading{Homological complexity.}
Let $k \geq 1$ be an integer or $\infty$ and $\mathcal F$ a family of sets in a topological space~$\Rspace$. 
 We use the notation $\bigcap \F := \bigcap_{F \in \F} F$ as a shorthand for the intersection of a family of sets.
We consider singular homology with coefficients in $\Z_2$ and denote by $\tilde\beta_i(X; \Z_2)$ the $i$th reduced Betti number of $X \subseteq \Rspace$.
We define the \emph{$k$-level homological complexity} of $\F$ as:

\[
 \sup\left\{\widetilde\beta_i\left(\bigcap \G;\Z_2\right) \colon \emptyset \neq \G \subseteq \F, 0\leq i<k\right\}
\] 

and denote it by $HC_k(\F)$. We call the number $HC_\infty(\F)$ the \emph{(full) homological complexity}. 
\medskip
The definition makes sense also for other rings of coefficients. Note that even if the family $\F$ lives in $\R^d$, it may happen that $\tilde\beta_{k}(\bigcap \G; \mathbb Q) > 0$ for $k > d$ and some $\G \subsetneq \F$,
see \cite{Barratt_Milnor}, albeit we do not know whether this can also happen for $\Z_2$-coefficients.
In any case, for the reasons explained in Section~\ref{ss:Z2-coefficients}, we use the homology with $\Z_2$-coefficients throughout the paper.\\

Many families we are used to work with have bounded homological complexity. To name a few:
family of convex sets in $\R^d$, families of open sets in a topological space, where the intersection of each subfamily is either empty or homologically trivial (such families are called \emph{good covers}), families of spheres and pseudospheres in $\R^d$, families of algebraic sets in $\R^d$ defined by polynomials of degree at most $D$ (By Milnor \cite{Milnor}, the Betti numbers of any algebraic set in $\R^d$ can be upper bounded by a function depending only on $d$ and $D$. Specifically, the function does not depend on the number of defining polynomials.). 
Another example can be obtained as follows.
We consider a family of polytopes in $\R^d$ whose all normals lie in some fixed finite set $S$, together with their polyhedral subcomplexes. It was shown in \cite[Corollary 5]{hb17} that homological complexity of such a family is bounded by a function of $|S|$ and $d$. As a special case we get a family of axis-alligned hollow boxes.\\

We can now state our main theorem.
\begin{theorem}[Bounded mid-level homological complexity implies Radon]\label{t:boundedRadon}
For every non-negative integers $b$ and $d$ there is a number $r(b,d)$ such that the following holds:
 If $\mathcal F$ is a family of sets in $\R^d$ with $HC_{\lceil d/2 \rceil}(\F)\leq b$, then $r(\mathcal F)\leq r(b,d)$.
\end{theorem}

We note that the bound $r(b,d)$ that our proof provides is very large (but still elementary recursive) as it relies on successive applications of Ramsey's theorem.
 On the other hand, the bound on $r(\F)$ is qualitatively sharp in the sense that all $\widetilde \beta_i$, $0 \leq i < \lceil \frac{d}{2}\rceil$, need to be bounded in order to have a universal bound on the Radon number. We discuss this in Section~\ref{s:discussion}, Example \ref{ex:sharpness}.

\subsection{Embeddability}\label{s:embeddability}
We have seen that for a family of sets in $\R^d$, in order to have a bounded Radon number, it suffices to restrict the reduced Betti numbers up to dimension $\lceil d/2 \rceil -1$. Which Betti numbers do we need to restrict, if we replace $\R^d$ by some other topological space $\Rspace$?  The following paragraphs provide some simple bounds if $\Rspace$ is a simplicial complex or a smooth real manifold.
The base for the statements is the following simple observation: Given a topological space $\Rspace$ embeddable into $\R^d$, we may view any subset of $\Rspace$ as a subset of $\R^{d}$ and use Theorem~\ref{t:boundedRadon}. 

Since  any (finite) $k$-dimensional simplicial complex embeds into $\R^{2k+1}$, we have: 
\begin{itemize}
 \item  If $K$ is a (finite) $k$-dimensional simplicial complex and $\F$ is a family of sets in $K$ with $HC_{k+1}(\F)\leq b$, then $r(\F)\leq r(b,2k+1)$. 
\end{itemize}
This result is again qualitatively sharp, i.e. all $\tilde \beta_i,  0 \leq i \leq k$, need to be bounded. In fact, it is enough to slightly modify  Example \ref{ex:sharpness} to show this. Namely, we consider the ambient space to be the union of the stated sets (taken as a simplicial complex) and not $\mathbb R^d$.

Using the strong Whitney's embedding theorem~\cite{whitneyimbedding}, stating that any smooth real $k$-dimensional manifold embeds into $\R^{2k}$, we obtain the following:
\begin{itemize}
  \item If $M$ is a smooth $k$-dimensional real manifold and $\mathcal F$ is a family of sets in $M$ with $HC_{k}(\F)\leq b$, then $r(\F)\leq r(b,2k)$.
 \end{itemize}
 
 Unlike in the previous statements we do not know whether bounding all reduced Betti numbers $\widetilde \beta_i$, $0 \leq i \leq k-1$, is necessary. The following result about surfaces indicates that it possibly suffices to bound less.
 Let $\F$ be a family of sets in a surface $S$, where by a \emph{surface} we mean a compact two-dimensional real manifold. 
In order to have a bounded Radon number $r(\F)$, it is enough to require that $HC_{1}(\F)$ is bounded,
that is, it only suffices to have a universal bound on the number of connected components.
\begin{theorem}\label{t:surf}
 For each surface $S$ and each integer $b\geq 0$ there is a number $r_S(b)$ such that each family $\F$ of sets in $S$ satisfying $HC_{1}(\F)\leq b$ 
 has $r(\F)\leq r_S(b)$.
\end{theorem}
See Section~\ref{s:CCHM} for the proof. It would be very interesting to see what happens for manifolds of dimension three and more.

\begin{problem}
 Given a $d$-dimensional manifold $M$, 
decide whether $r(\mathcal F)$ is bounded for all families $\F \subseteq M$
with bounded $HC_{\lceil d/2 \rceil}(\F)$.
\end{problem}
Regarding the surface, the bound on $r_S(b)$ that our proof provides is again very large. If the surface is connected, the bound depends on $b$ and the Euler characteristics of $S$, for disconnected surfaces one has to compute the bound for each connected component separately and then take the maximum.

\begin{problem}
 For a connected surface $S$, is $r_S(b)$ polynomial in the Euler characteristics of~$S$?
\end{problem}

\subsection{Consequences and related results} \label{s:direct_conseq}
There are many useful parameters related to the Radon number. Let us list some of them. 
A family $\mathcal F$ has \emph{Helly number} $h(\mathcal F)$, if $h(\mathcal F)$ is the smallest integer $h$ with the following property:
If in a finite subfamily $\mathcal S\subseteq \mathcal F$ each $h$ members of $\mathcal S$ have a point in common, then all the sets of $\mathcal S$ have a point in common.
If no such $h$ exists, we put $h(\F) = \infty$.

A direct generalization of Radon numbers are the \emph{Tverberg numbers} (sometimes also called \emph{partition numbers}). Given an integer $k\geq 3$, we say that $\F$ has \emph{$k$th Tverberg number} $r_k(\F)$, if $r_k(\F)$ is the smallest integer $r_k$ such that any set $S \subseteq X$ of size $r_k$ can be split into $k$ pair-wise disjoint parts $S=P_1\sqcup P_2\sqcup\ldots \sqcup P_k$
satisfying $\bigcap_{i=1}^k \operatorname{conv}_{\mathcal F}P_i\neq \emptyset$. We set $r_k(\mathcal F)=\infty$ if there is no such $r_k$.

By older results\footnote{Some of these implications were formulated in the setting of \emph{convexity spaces}. Nevertheless, we show in Section~\ref{ss:convexvsclosed} that they also hold for arbitrary \emph{closure spaces}, which are in a direct correspondence with our setting. (We also provide definitions of convexity and closure spaces in that section).},
 bounded Radon number implies bounded Helly number \cite{Levi1951} as well as bounded Tverberg numbers  
 \cite[(6)]{jamison1981},  where the later bound was several times improved since then \cite{partition_conjecture, Bukh, domotor}. 
Combining these theorems with Theorems \ref{t:boundedRadon} and \ref{t:surf}, respectively, yields
 that constant mid-level homological complexity implies bounded Helly and Tverberg numbers, respectively.
From these four corollaries only the result that for sets in $\R^d$ bounded $HC_{\lceil d/2 \rceil} $ implies bounded Helly number has been shown earlier~\cite[Theorem 1]{hb17}.

Due to a recent result by Holmsen and Lee \cite[Theorem 1.1]{boundedRadon_fractHelly}, convexity spaces with bounded Radon number satisfy a fractional Helly property. 
Thus, Theorem \ref{t:boundedRadon} combined with \cite[Theorem 1.1]{boundedRadon_fractHelly} immediately gives the following fractional Helly theorem (we refer to the discussion in Section~\ref{ss:convexvsclosed} regarding the connection between our setting and convexity spaces).
\begin{theorem}\label{t:non-optimal-frachelly}
For every integers $b, d \geq 0$ and for every $\alpha \in (0,1)$ there exist  $\beta=\beta(\alpha,b,d) > 0$ and an integer $m = m(b,d)$ with the following property.
Let $\F$ be a family of sets in $\R^d$ with $HC_{\lceil d/2\rceil}(\F)\leq b$. Let $\mathcal G$ be a finite subfamily of $\F$ of $n \geq m$ sets. If  at least $\alpha\binom{n}{m}$ of the $m$-tuples of $\G$ have non-empty intersection, then there are at least $\beta n$ members of $\G$ whose intersection is non-empty.
 \end{theorem}
 
  The minimal number $m$ for which the theorem holds is called the \emph{fractional Helly number} of $\F$. That is, a family $\F$ has  fractional Helly number $m$ if $m$ is the minimal integer with the following property: for every $\alpha\in (0,1)$, there is $\beta=\beta(\alpha)>0$ such that the conclusion of Theorem \ref{t:non-optimal-frachelly} holds, i.e. given an $n$-element subfamily $\G$ of a family $\F$, if at least $\alpha\binom{n}{m}$ of the $m$-tuples of $\G$ have non-empty intersection, then there are at least $\beta n$ members of $\G$ whose intersection is non-empty.
\medskip

 As mentioned in the previous section, we may use Theorem \ref{t:boundedRadon} also for any topological space $\Rspace$ embeddable into $\R^d$ (e.g smooth real $d$-dimensional manifolds or finite simplicial complexes). This in turn provides corresponding fractional Helly theorems.  Grassmanians or flag manifolds can serve as examples of manifolds often encountered in geometry.
Using our tools we can e.g. obtain some fractional Helly-type statements for line-transversals.
Combining \cite[Theorem 1.1]{boundedRadon_fractHelly} with Theorem \ref{t:surf}, we immediately get a fractional Helly theorem for surfaces, which we formulate here for a further reference:

  \begin{theorem}\label{t:fract_helly}
  Let $S$ be a surface. Then
for every integer $b \geq 0$ and for every $\alpha \in (0,1)$ there exist $\beta=\beta(\alpha, b, S)$ and an integer $m=m(b,S)$ such that the following holds. 
Any family $\F$ of sets in $S$ with $HC_{1}(\F)\leq b$ has the fractional Helly number at most $m$.
 \end{theorem}

The existence of a fractional Helly theorem for sets with bounded homological complexity
might be seen as the most important application of Theorem \ref{t:boundedRadon}, not only because it implies an existence of weak $\eps$-nets and a $(p,q)$-theorem (which we both discuss later), but also in its own right. Its existence
answers positively a question by Matou\v sek (personal communication), also posed in \cite[Open Problem 3.6]{survey}. 

  We note that the bound $m(b,d)$ on the fractional Helly number we obtain from the proof is not optimal.
 So what is the optimal bound?
 The case of $(d-1)$-flats in $\R^d$ in general position shows that we cannot hope for anything better than $d+1$.
  In Section~\ref{s:pq_surface} we establish a reasonably small bound for a large class of families $\F$ of open subsets of surfaces  using  the result of the author and Kalai \cite{planar_sets}. In particular, for families $\F$ of open sets with $HC_1(\F) = 0$, we obtain the optimal bound.
 
 \begin{theorem}[fractional Helly for surfaces]\label{t:frachelly_surfaces}
  Let $b \geq 0$ be an integer. We set $m=3$ for $b=0$ and $m=2b+4$ for $b \geq 1$, respectively. Then for any surface $S$ and $\alpha \in (0,1)$ there exists $\beta = \beta(\alpha, b,S)>0$ with the following property. Let $\A$ be a family of $n$ open subsets of a surface $S$ with $HC_1(\A) \leq b$. If at least $\alpha \binom{n}{m}$ of the $m$-tuples of $\A$ are intersecting, then there is intersecting subfamily of $\A$ of size at least $\beta n$.
 \end{theorem}
It might be surprising that the bound on the fractional Helly number  does not depend on the surface $S$. What depends on $S$ is the actual value of $\beta$: This value depends on $\alpha$, $b$, and on the number of connected components of $S$ together with their Euler characteristics.

 We note that the statement of Theorem \ref{t:frachelly_surfaces} also holds for finite families of open sets in $\R^2$, 
 since the plane can be seen as an open subset of a 2-dimensional sphere. The proof of Theorem \ref{t:frachelly_surfaces} is given in Section~\ref{s:pq_surface}.
 
 We conjecture that the fractional Helly number of a family $\F$ 
 is independent of the homological complexity of $\F$, provided that it is bounded, and depends only on the topological space itself.
 
 \begin{problem}
Is the fractional Helly number in Theorems \ref{t:non-optimal-frachelly} and \ref{t:frachelly_surfaces}, respectively, independent of $b$? What can be said for manifolds?
 \end{problem}

 \heading{$(p,q)$-theorems.}
 Another important corollaries of Theorems \ref{t:boundedRadon} and \ref{t:surf} are so called $(p,q)$-theorems, which are, in fact, generalizations of Helly theorem.
  We say that a family $\mathcal F$ of subsets of $X$ has the \emph{$(p,q)$-property},
 if among every $p$ sets of $\F$, some $q$ have a point in common. The standard $(p,q)$-theorem \cite{alon-kleitman} says that for any finite family $\mathcal F$ of convex sets in $\mathbb R^d$ satisfying the $(p,q)$-property, a constant number of points in $\mathbb R^d$ is sufficient to intersect all sets in $\mathcal F$.
 
 It is well-known that an \emph{intersection closed} set system with a fractional Helly number $m$ in turn provides a $(p,q)$-theorem for $p \geq q \geq m$, see \cite[Theorems 8(i) and 9, and the discussion in §2.1]{transversal-hypergraph}. Thus, in order to show that a finite family $\mathcal F$ with bounded mid-level homological complexity enjoys a $(p,q)$-theorem, it is enough to apply Theorem \ref{t:non-optimal-frachelly} to the family $\mathcal F^{\cap}:=\{\bigcap_{S \in \mathcal G} S: \mathcal G \subseteq \mathcal F\}$ and observe that $\mathcal F^\cap$ has bounded mid-level homological complexity if and only if $\mathcal F$ does. Similarly, if $\mathcal F$ is a finite family of open subsets of a surface, we apply Theorem \ref{t:frachelly_surfaces} to $\mathcal F^\cap$ and obtain the following theorem.
 
 \begin{theorem} \label{t:pq_surface}
 Let $b \geq 0$ be an integer. Set $m=3$ for $b=0$ and $m=2b+4$ for $b \geq 1$, respectively.
  For any integers $p \geq m$ and a surface $S$, there exists an integer $C=C(p,b,S)$ such that the following holds. Let $\F$ be a finite family of open subsets of $S$ with $HC_1(\F) \leq b$. If $\F$ has the $(p,m)$-property, then there is a set of points that intersects all sets from $\F$ and has at most $C$ elements.
 \end{theorem}
 
 We note that the case $b=0$ in Theorem \ref{t:pq_surface} settles a conjecture by Holmsen, Kim, and Lee \cite[Conj. 5.3]{nerves_minors}.
 \begin{remark}
 A careful reader might notice a slightly different formulation of the conjecture \cite[Conj. 5.3]{nerves_minors} involving a parameter $q$: any finite family of open sets with trivial $HC_1(\F)$ and satisfying the $(p,q)$-property, $p\geq q\geq 3$, can be pierced by at most $C=C(p,q,S)$ elements. However, as noticed already in \cite{alon-kleitman}, since we are only interested in the existence of the constant $C(p,q,S)$ and not in its precise value, it is sufficient to consider the case $q=3$.                                                                                                                                       \end{remark}
 
 \heading{Weak $\eps$-nets.}
 Another important landmark in combinatorial convexity is a concept of \emph{weak $\varepsilon$-nets}. As shown in \cite{transversal-hypergraph}, they are closely related to fractional Helly property. We recall the definition first.
 Given a family $\F$ of subsets of $X$, $\varepsilon > 0$ and a finite subset $Y \subseteq X$, we say that $N \subseteq X$ is a \emph{weak $\varepsilon$-net} for $Y$ (with respect to $\F$) if it intersects every set from $\mathcal F$ containing at least $\varepsilon |Y|$ points of $Y$. One of the main problems in discrete geometry is to find weak $\varepsilon$-nets of small size, in particular when $\mathcal F$ is a family of convex sets in $\mathbb R^d$.

 Our Theorem \ref{t:non-optimal-frachelly} together with \cite[Theorem 9 and the discussion in §2.1]{transversal-hypergraph}
imply that for any integers $b \geq 0, d \geq 1$, there are positive constants $c_1$ and $c_2$ such that any family of subsets of $\mathbb R^d$ with $\lfloor d/2\rfloor$-level homological complexity bounded by $b$ admits a weak $\varepsilon$-net of size $c_1/\varepsilon^{c_2}$, where $c_1, c_2$ depends only on $b$ and $d$ (and the fractional Helly number $m$, but this is a function of $b$ and $d$). 
We refer to \cite[Thm 4.1]{boundedRadon_fractHelly} for a closely related result stating that in a convexity space with Radon number at most $r$ any finite subset of the underlying set admits a weak $\varepsilon$-net of size $c_1/\varepsilon^{c_2}$, where both $c_1, c_2$ are positive constants depending on $r$.

 \heading{Carath\'eodory number.}
We have seen that bounded homological complexity has many interesting consequences.
However, not all parameters of $\mathcal F$ can be bounded by the homological complexity alone. We show that the Carath\'eodory number is one such example.
A family $\mathcal F$ has \emph{Carath\'eodory number} $c(\mathcal F)$, if $c(\mathcal F)$ is the smallest integer $c$ with the following property:
For any set $S\subseteq X$ and any point $x\in\operatorname{conv}_{\mathcal F}(S)$, there is a subset $S'\subseteq S$ of size at most $c$ such that $x\in\operatorname{conv}_{\mathcal F}(S')$. 
If no such $c$ exists, we put $c(\F) = \infty$.

The following theorem shows that it is easy to construct an example of a finite $\mathcal F$ of trivial full-level homological complexity with arbitrarily high Carath\'eodory number. 

\begin{theorem}[Bounded homological complexity does not imply Carath\'eodory]\label{t:no_carat}
 For every positive integers $c\geq2$ and $d\geq2$ there is a finite family $\mathcal F$ of sets in $\R^d$ of full-level homological complexity zero, satisfying $c(\mathcal F)=c$.  
\end{theorem}

\noindent
\begin{minipage}{0.75\textwidth}
 \begin{proof}
Indeed, consider a star with $c$ spikes $T_1,T_2,\ldots, T_c$
each containing a point $t_i$.
Let $A_i:=\bigcup_{j\neq i}T_j$
and $\F=\{A_1,A_2,\ldots, A_c\}$.

Then any intersection of the sets $A_i$
is contractible ( as the center $s$ of the star is contained in each $A_i$), and hence topologically trivial. Let $S=\{t_1,\ldots,t_c\}$. Observe that $\conv_\F S=\R^d$. Let $x$ be any point in $(\conv_\F S) \setminus \bigcup_{i=1}^c A_i$. Then $x\in\conv_F S$, and $x\notin\conv_F S'$ for any $S'\subsetneq S$. Thus $c(\A) = c$. 
\end{proof}
\end{minipage}
\begin{minipage}{0.25\textwidth}
  \flushright{
  \includegraphics[page=1]{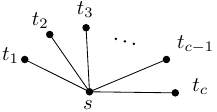}
   }
 \end{minipage}
\newline

\heading{Organization of the paper.}
Our main result, Theorem \ref{t:boundedRadon}, as well as its surface version, Theorem \ref{t:surf}, are proven in Section~\ref{s:technique}. In Section~\ref{s:pq_surface} we show that for families of open sets on a surface (or a plane), the fractional Helly number is linear in $b$ (Theorem~\ref{t:frachelly_surfaces}). Some further open problems are formulated in Section~\ref{s:discussion}. That section also contains a remark about computing with $\Z_2$-coefficients or qualitative sharpness of Theorem \ref{t:boundedRadon} as well as a discussion about convexity spaces.

\section{Technique}\label{s:technique}

We generalize and polish a technique that was originally developed in~\cite{hb17}, for showing an upper bound on the Helly number for families in $\R^d$ with bounded mid-level homological complexity. 
Independently of this strengthening we also separate the combinatorial and topological part of the proof, which we hope will be useful in the future. Prospective benefits of this step are explained in the final discussion in Section~\ref{s:discussion}.

We start with the topological tools (Sections \ref{s:HAE} and \ref{s:CCHM}) including the proof of Theorem~\ref{t:boundedRadon} modulo Proposition~\ref{p:ccm}. We divide the proof 
of the main ingredient (Proposition~\ref{p:ccm}) into two parts: Ramsey-type result (Section \ref{s:ramsey}) and induction (Section~\ref{s:induction}).

\heading{Notation \& convention.}
For an integer $n \geq 1$, let $[n] = \{1,\ldots,n\}$. 
If $P$ is a set, we use the symbol $2^P$ to denote the set of all its subsets and $\binom{P}{n}$ to denote the family of all $n$-element subsets of $P$.
We denote by $\Delta_n$ the standard $n$-dimensional simplex. If $K$ is a simplicial complex, $V(K)$ stands for its set of vertices and $K^{(k)}$ stands for its $k$-dimensional skeleton, i.e. the subcomplex formed by all its faces of dimension up to $k$. 
The \emph{barycentric subdivision} $\sd K$ of an abstract simplicial complex $K$ is the complex formed by all the chains contained in the partially ordered set $(K\setminus\{\emptyset\},\subseteq)$, so called the \emph{order complex} of $(K\setminus\{\emptyset\},\subseteq)$.
Let $v$ be a singleton disjoint from vertices of $K$. The \emph{cone} $v * K$ of $K$ with an appex $v$ is an abstract simplicial complex $\{v \cup \sigma \colon \sigma \in K \} \cup \{\emptyset\}$.
Unless stated otherwise, we only work with abstract simplicial complexes.
All chain groups and chain complexes are considered with $\Z_2$-coefficients.

\subsection{Homological almost embeddings}\label{s:HAE}
Homological almost embeddings are the first ingredient that we need. Before defining them, let us first recall (standard) almost-embeddings. Let $\mathbf{R}$ be a non-empty topological space.
\begin{definition}\label{d:alme}
 Let $K$ be an (abstract) simplicial complex with geometric realization $|K|$.
 A continuous map $f\colon |K|\to \mathbf{R}$ is an \emph{almost-embedding} of $K$ into $\mathbf{R}$, if the images of disjoint simplices are disjoint.
\end{definition}

\begin{definition}[\cite{hb17}]\label{d:homrep}
Let $K$ be a simplicial complex, and consider
a chain map
$\gamma\colon C_\ast(K; \Z_2)\rightarrow
C_\ast(\mathbf{R}; \Z_2)$ from the simplicial chains in $K$ to singular chains in
$\mathbf{R}$.

\begin{enumerate}[label=(\roman*)]
\item The chain map $\gamma$ is called \emph{\nontrivial} if the image of every vertex of $K$ is a finite set of points
  in~$\mathbf{R}$ \textup{(}a 0-chain\textup{)} of \emph{odd} cardinality. 
  \item The chain map $\gamma$ is called a  \emph{homological
    almost-embedding} of $K$ into $\mathbf{R}$ if it is \nontrivial\ and if, additionally, the following holds: whenever $\sigma$ and $\tau$ are disjoint simplices of $K$, their image chains
  $\gamma(\sigma)$ and $\gamma(\tau)$ have disjoint supports, where
  the support of a chain is the union of (the images of) the singular simplices 
  with nonzero coefficient in that chain.
\end{enumerate}
\end{definition}
For an illustration, see Figure \ref{f:almost_embed}.

\begin{remark}
 If we consider augmented chain complexes with chain groups also in dimension $-1$, 
  then being \nontrivial\ is equivalent to requiring that $\gamma_{-1}\colon C_{-1}(K;\Z_2)\to C_{-1}(\mathbf{R};\Z_2)$ is an isomorphism.
\end{remark}

We note that any continuous map $f\colon |K| \to \Rspace$ induces a nontrivial chain map and that if $f$ is an almost-embedding, the induced chain map is a homological almost-embedding. 

\begin{figure}
  \begin{center}
  \includegraphics[page=2]{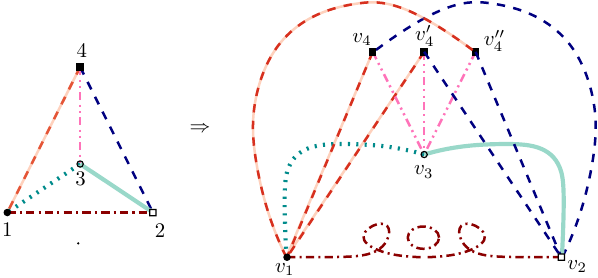}
  \caption{An example of a homological almost-embedding of $K_4$ into the plane.}\label{f:almost_embed}
 \end{center}
 \end{figure}

The next ingredient we need is the fact that
the classical result about the non-almost-emebddability of the $\lceil d/2\rceil$-skeleton of $(d+2)$-dimensional simplex into $\R^{d}$ holds also in the homological setting.

\begin{theorem}[Corollary 14 in~\cite{hb17}]\label{c:nohomrep}\label{t:homological_van_Kampen}
  For any $d \ge 0$, the $\lceil d/2\rceil$-skeleton $\skelsim{\lceil d/2\rceil}{d+2}$ of the
  $(d+2)$-dimensional simplex has no homological almost-embedding in
  $\R^{d}$.
\end{theorem}
 We note that the proof combines the standard cohomological proof that $\skelsim{k}{2k+2}$ does not almost-embed into $\R^{2k}$  with the fact that cohomology ``does not distinguish'' between maps and non-trivial chain maps. For details see~\cite{hb17}.\\

\subsection{Constrained chain maps}~\label{s:CCHM}
We strengthen the machinery from~\cite{hb17} in order to capture our more general setting.
To prove Theorem \ref{t:boundedRadon}, we need one more definition (Definition \ref{d:constr_chain_map}). A curious reader may compare our definition of constrained chain map with the definition from \cite[Section 3.2]{hb17}. For convenience of the reader we state the older definition as Definition~\ref{def:constraned_chain_map_old}. Let us remark that the definition introduced in the current paper is more versatile. (Although it might not be obvious on the first sight.) Unlike the previous definition, the current form allows us to prove the bound on the Radon number. We discuss the differences in more detail in Remark~\ref{r:constrained_chain_maps}.\medskip

Let $\mathbf{R}$ be a topological space, let $K$ be a simplicial complex and let $\gamma: C_*(K) \to C_*(\mathbf{R})$ be a chain map from the simplicial chains of $K$ to the singular chains of $\mathbf{R}$.

\begin{definition}[Constrained chain map]\label{d:constr_chain_map}
 Let $\mathcal F$ be a family of sets in $\mathbf{R}$ and $P$ be a set of points in $\mathbf{R}$.
 Let $\gamma: C_*(K) \to C_*(\mathbf{R})$
be the aforementioned chain map. We say that $\gamma$ is \emph{constrained by
  $(\F,\Phi)$} if there exists a map $\Phi$ from $K$ to $2^{P}$ such that

\begin{enumerate}[label=(\roman*)]
 \item $\Phi(\sigma   \cap \tau) = \Phi(\sigma) \cap \Phi(\tau)$ for all $\sigma, \tau \in K$ and
  $\Phi(\emptyset)=\emptyset$. \label{it:(i)}

\item For any simplex $\sigma \in K$, the support of
  $\gamma(\sigma)$ is contained in $\conv_\mathcal F{\Phi(\sigma)}$.\label{it:(ii)}
\end{enumerate}

If there is some $\Phi$ such that the chain map $\gamma$ 
is constrained by $(\F,\Phi)$,
we say that $\gamma$ is \emph{constrained by $(\F, P$)}.
\end{definition}

\begin{remark}\label{r:monotonicity}
 Note that the map $\Phi$ from Definition \ref{d:constr_chain_map} is monotone.
 Indeed, for any simplices $\tau\subseteq\sigma \in K$, the item \ref{it:(i)} implies
   $
    \Phi(\tau) = \Phi(\tau \cap \sigma) = \Phi(\tau) \cap \Phi(\sigma) \subseteq \Phi(\sigma).
   $
\end{remark}

We now relate constrained chain maps and homological almost embeddings.

\begin{lemma}\label{l:hom_almost_embedding}
 Let $\gamma: C_*(K) \to C_*(\mathbf R)$ be a nontrivial chain map constrained by $(\mathcal F, P)$. Then $\gamma$ is a homological almost embedding or there exist two disjoint sets $A, B \subseteq P$ with $\conv_{\mathcal F} A \cap \conv_{\mathcal F}B \neq \emptyset$.
\end{lemma}

\begin{proof}
 Let us assume that for every two disjoint subsets $A,B \subseteq P$ we have $\conv_{\mathcal F}A \cap \conv_{\mathcal F} B = \emptyset$. As $\gamma$ is non-trivial, it is enough to show that images of disjoint faces in $K$ have disjoint supports.
 Let $\sigma$ and $\tau$ be two disjoint simplices of $K$ and 
 $\Phi: K \to 2^P$ be a map witnessing that $\gamma$ is constrained by $(\mathcal F, P)$. By Definition \ref{d:constr_chain_map}\ref{it:(i)}, the sets  
 $\Phi(\sigma)$ and $\Phi(\tau)$
 are disjoint.  Furthermore, the supports of $\gamma(\sigma)$ and $\gamma(\tau)$ are contained, respectively, in 
 $\conv_\mathcal{F}\Phi(\sigma)$ and $\conv_\mathcal{F}\Phi(\tau)$ by Definition \ref{d:constr_chain_map}\ref{it:(ii)}. However, by  assumption 
 $
  \conv_\F \Phi(\sigma) \cap \conv_\F \Phi(\tau)= \emptyset
 $ and the result follows.
\end{proof}

The most important ingredient for the proof of Theorem \ref{t:boundedRadon} is the following proposition:  
\begin{proposition}\label{p:ccm}
  For any finite simplicial complex $K$ and a non-negative integer $b$
  there exists a constant $r_K(b)$ such that the following holds. For
  any  family $\F$ in a topological space~$\mathbf{R}$ with $HC_{\dim K}(\F)\leq b$ and any set $P \subseteq \mathbf{R}$ of at least $r_K(b)$ points
  there exists a nontrivial chain map $\gamma: C_*(K) \to
  C_*(\mathbf{R})$ that is constrained by $(\F,P)$.
  
  Furthermore, if $\dim K\leq 1$, one can even find such $\gamma$ that is induced by some continuous map $f \colon |K| \to \mathbf{R}$ from the geometric realization $|K|$ of $K$ to $\mathbf{R}$.
\end{proposition}  

Before proving Theorems \ref{t:boundedRadon} and \ref{t:surf}, let us relate Proposition~\ref{p:ccm} to the Radon number.

\begin{proposition}\label{p:gen}
 Let $\mathbf{R}$ be a topological space and $K$ a finite simplicial complex that does not homologically almost-embed into $\mathbf{R}$.
 Then for each integer $b\geq 0$ and every family
 $\F$ of sets in $\mathbf{R}$ satisfying $HC_{\dim K}(\F)\leq b$, the Radon number $r(\F)$ is at most $r_K(b)$, where $r_K(b)$ is the constant from Proposition~\ref{p:ccm}.
 
 Moreover, if $\dim K\leq 1$, it suffices to assume that $K$ does not almost-embed into $\mathbf{R}$.
\end{proposition}
\begin{proof}[{Proof of Proposition \ref{p:gen} assuming Proposition \ref{p:ccm}}]
 If $r(\F) > r_K(b)$, then there is a set $P$ of $r_K(b)$ points such that for any two disjoint subsets $P_1,P_2\subseteq P$ we have $\conv_{\F}(P_1)\cap\conv_{\F}(P_2)=\emptyset$.
 Let $\gamma\colon C_*(K)\to C_*(\mathbf{R})$ be a nontrivial chain map constrained by $(\F,P)$ given by Proposition~\ref{p:ccm}.
 By Lemma \ref{l:hom_almost_embedding}, $\gamma$ is a homological almost-embedding of $K$, which is a contradiction.
 
 If $\dim K\leq 1$, one can take $\gamma$ to be induced by a continuous map $f \colon |K| \to \mathbf{R}$. However, one can easily check that in that case $\gamma$ is a homological almost-embedding if and only if $f$ is an almost-embedding.  
\end{proof}

Theorems~\ref{t:boundedRadon} and ~\ref{t:surf} are now immediate consequences of Proposition \ref{p:gen}.
 \begin{proof}[Proof of Theorem \ref{t:boundedRadon}]
 By Theorem~\ref{c:nohomrep}, $\skelsim{\lceil d/2\rceil}{d+2}$ does not homologically almost-embed into $\R^d$, so Proposition~\ref{p:gen} applies and yields Theorem~\ref{t:boundedRadon}.
 \end{proof}

 \begin{proof}[Proof of Theorem~\ref{t:surf}]
  By results in~\cite{kuhnel_our}, for each surface $S$ there is a finite graph $G$ that does not almost-embed into $S$, so Proposition~\ref{p:gen} applies. 
 \end{proof}

 We note that compared to \cite{kuhnel_our}, recent works by Pat\'ak, Tancer \cite{Patak_Tancer}, and Fulek, Kyn\v cl \cite{rk:genus} provide much smaller graphs that are not almost-embeddable into~$S$.
 
 \begin{remark}\label{r:constrained_chain_maps}
  Let us compare a constrained chain map defined in this paper (Definition~\ref{d:constr_chain_map}) with the notion of constrained chain map defined in \cite[Section 3.2]{hb17}. We recall the definition from \cite{hb17}:
  
 \begin{definition}\label{def:constraned_chain_map_old}
 Let $\mathcal F = \{U_1, \ldots, U_n\}$ be a family of sets in a topological space $\mathbf R$, let $K$ be a simplicial complex and let $\gamma: C_*(K) \to C_*(\mathbf{R})$ be a chain map from the simplicial chains of $K$ to the singular chains of $\mathbf{R}$.
  We say that $\gamma$ is \emph{constrained by $\mathcal F$} if there exists a map $\Phi'$ from $K$ to $2^{[n]}$ such that
 
 \begin{enumerate}[label=(\roman*)]
  \item  $\Phi'(\sigma   \cap \tau) = \Phi'(\sigma) \cap \Phi'(\tau)$ for all $\sigma, \tau \in K$ and
   $\Phi'(\emptyset)=\emptyset$.
 
 \item\label{i:ii} For any simplex $\sigma \in K$, the support of
   $\gamma(\sigma)$ is contained in $\bigcap_{i \in [n]\setminus \Phi'(\sigma)}U_i$.
 \end{enumerate}%
  \end{definition}

  There are two obvious differences between Definitions \ref{d:constr_chain_map} and \ref{def:constraned_chain_map_old}. First, the family $\F$ in Definition \ref{def:constraned_chain_map_old} is finite, and second, there is no set $P$ (we work with $[n]$ instead).  Moreover, in Definition \ref{def:constraned_chain_map_old}\ref{i:ii}
  the support of $\gamma(\sigma)$ is required to be contained in $\bigcap_{i \in [n]\setminus \Phi'(\sigma)} U_i$ instead of in $\conv_\mathcal F \Phi(\sigma)$.

  \begin{figure}
   \begin{center}
   \includegraphics[page=3]{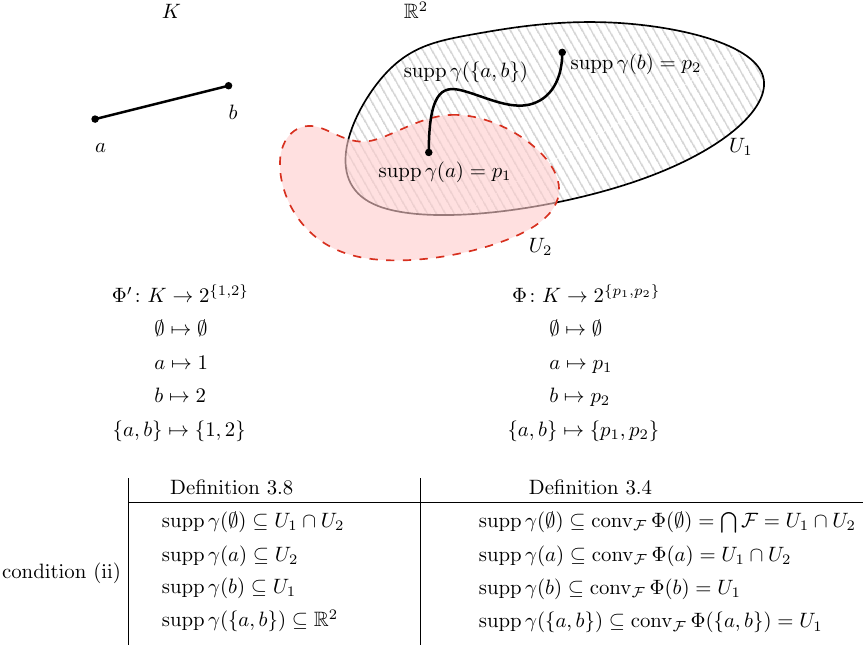}
   \caption{Let $K$ be a simplicial complex formed by a single edge with vertices $a,b$ and $\gamma$ be a chain map induced by a geometric realization of $K$ in $\mathbb R^2$. The family $\mathcal F = \{U_1, U_2\}$ of subsets of $\mathbb R^2 $ is indicated on the picture as well as the map $\Phi'$. It is easy to verify that $\gamma$ is constrained by $\mathcal F$ according to Definition \ref{def:constraned_chain_map_old}.
      We denote by $p_1, p_2$ the images of the vertices of $K$ in the geometric realization, respectively, and put $P = \{p_1, p_2\}$. Again, it is easy to verify that $\gamma$ is constrained by $(\mathcal F, P)$ according to Definition \ref{d:constr_chain_map}, with $\Phi$ defined at the picture.
      Finally, $U_1 \cap U_2 =\conv_\mathcal F \Phi(a) \subsetneq \bigcap_{i \in [2]\setminus \Phi'(a)} = U_2$.}\label{fig:constr_chain_map}
     \end{center}
  \end{figure}

   We now describe a situation when $\bigcap_{i \in [n]\setminus \Phi'(\sigma)} U_i = \conv_\mathcal F \Phi(\sigma)$.
  Let $\F = \{U_1, \ldots, U_n\}$ be a family of sets such that $\bigcap \mathcal F = \emptyset$ and $\bigcap (\mathcal F \setminus \{U\}) \neq \emptyset$ for every   $U \in \mathcal F$. We set $P := \{ p_i \colon p_i \in  \bigcap_{j \neq i} U_j\}$ and let $\Phi: K \to 2^P$ be some map. Let us refer to this setting as (S)\label{s:setting}. Note that $|P| = n$ (otherwise $\bigcap \mathcal F \neq \emptyset$) and there is a bijection $\pi: P \to [n]$ given by $p_i \mapsto i$. The bijection $\pi$ naturally extends to the bijection $\pi': 2^P \to 2^{[n]}$ by $P_I \mapsto  I$, where $P_I := \{p_i \colon i \in I\}$.
  Note that $P_{I \cap J} = P_I \cap P_J$. We set $\Phi' = \pi' \circ \Phi$ and check that $\Phi'$ satisfies the condition~\ref{it:(i)} in Definition \ref{def:constraned_chain_map_old}, if and only if $\Phi$ satisfies the same condition. Regarding condition \ref{it:(ii)}, let $\sigma \in K$ and $\Phi(\sigma) = P_I$ for some $I \subseteq [n]$. Since $P_I \subseteq \bigcap_{i \in [n]\setminus I} U_i$ and $P_I \not\subseteq U_i$ for any $i \in I$, we get 
  \[\conv_\mathcal F \Phi(\sigma) = \conv_\mathcal F P_I =  \bigcap_{i \in [n]} \{U_i: P_I \subseteq U_i\}= \bigcap_{i \in [n]\setminus I} U_i = \bigcap_{i \in [n]\setminus \Phi'(\sigma)} U_i.\]
  We conclude that in the setting (S), a chain map $\gamma$ is constrained by $\mathcal F$ in the sense of Definition \ref{def:constraned_chain_map_old} if and only if it is constrained by $(\mathcal F, P)$ in the sense of Definition \ref{d:constr_chain_map}, with $P$ defined as above. We denote this equivalence by (E).
  Furthermore, it follows that in the setting (S) Proposition \ref{p:ccm} implies \cite[Proposition 30]{hb17}. We note that assuming (S) is not a serious restriction as this is exactly the setting in which Proposition 30 is further used in \cite{hb17} to prove its main result stating that bounded homological complexity implies a bounded Helly number (\cite[Theorem 1]{hb17}). In fact, Proposition \ref{p:ccm} together with Lemma~\ref{l:hom_almost_embedding} and the equivalence (E) stated above directly yield a proof of \cite[Theorem 1]{hb17}.
  
  On the other hand, if we are not in the setting (S), it can happen that 
  $\conv_\mathcal F \Phi(\sigma) \subsetneq \bigcap_{i \in [n]\setminus \Phi'(\sigma)} U_i$. Indeed, see Figure \ref{fig:constr_chain_map} describing an example of a chain map $\gamma$, a family $\mathcal F$, a point set $P$ and two maps $\Phi, \Phi'$ such that $\gamma$ is constrained by $\mathcal F$ according to Definition \ref{def:constraned_chain_map_old}, $\gamma$ is constrained by $(\mathcal F, P)$ according to Definition \ref{d:constr_chain_map} and 
  $\conv_\mathcal F \Phi(\sigma) \subsetneq \bigcap_{i \in [n]\setminus \Phi'(\sigma)} U_i$. Note that in this example $\bigcap \mathcal F \neq \emptyset$. 
  
  The absence of the point set $P$ in Definition \ref{def:constraned_chain_map_old} as well as no reasonable way how to conclude anything about $\conv_\mathcal F$ indicate why a constrained chain map defined via Definition \ref{def:constraned_chain_map_old} was insufficient for us to provide an upper bound on the Radon number.
  
 From now on, whenever we talk about constrained chain maps we refer to Definition~\ref{d:constr_chain_map}.
\end{remark}

\subsection{Combinatorial part of the proof}\label{s:ramsey}
The classical Ramsey theorem~\cite{ramsey30} states
that for all positive integers $k,n$ and $c$
there is a number $R_k(n;c)$ such that the following holds.
For each set $X$ satisfying $|X|\geq R_k(n;c)$ and each coloring $\rho\colon \binom{X}{k}\to [c]$,
there is a \emph{monochromatic} subset $Y\subseteq X$ of size $n$, where a subset $Y$ is monochromatic, if 
all $k$-tuples in $Y$ have the same color.
 We note that a coloring is just another name for a map. However, it is easier to say ``the color of $z$'', instead of ``the image of $z$ under $\rho$''.
Observe that the case $k=1$ corresponds to the pigeon hole principle and $R_1(n;c) = n(c-1)+1$.

We fix a set $X$ and positive integers $k, c$. In what follows we need to consider several colorings of $k$-element subsets of $X$. So suppose that for every $V\subseteq X$ we are given a coloring $\rho_V \colon \binom{V}{k}\to [c]$ of the $k$-element subsets of $V$. If $|V|<k$,  the coloring $\rho_V$ is, by definition, the empty map. See Figure \ref{fig:colorings} for an example.

\begin{figure}
\begin{center}
 I. \qquad $\begin{array}{c|*{16}{c}}
& \rho_\emptyset & \rho_a & \rho_b & \rho_c & \rho_d & \rho_{ab} & \rho_{ac} & \rho_{ad} & \rho_{bc} & \rho_{bd} & \rho_{cd} & \rho_{abc} & \rho_{abd} & \rho_{acd} & \rho_{bcd} & \rho_{abcd}\\ \hline
a & - & 2 & - & - & - & 1 & 2 & 1 & - & - & - & 2 & 2 & 2 & - & 1\\ 
b & - & - & 1 & - & - & 1 & - & - & 1 & 1 & - & 2 & 2 & - & 2 & 2\\ 
c & - & - & - & 3 & - & - & 2 & - & 2 & - & 1 & 1 & - & 1 & 2 & 1\\ 
d & - & - & - & - & 2 & - & - & 1 & - & 1 & 2 & - & 3 & 2 & 2 & 3\\  
 \end{array}
 $
 \vspace{1em}

  II. \quad $\begin{array}{c|*{16}{c}}
& \rho_\emptyset & \rho_a & \rho_b & \rho_c & \rho_d & \rho_{ab} & \rho_{ac} & \rho_{ad} & \rho_{bc} & \rho_{bd} & \rho_{cd} & \rho_{abc} & \rho_{abd} & \rho_{acd} & \rho_{bcd} & \rho_{abcd}\\ \hline
ab & - & - & - & - & - & 1 & - & - & - & - & - & 2 & 1 & - & - & 1\\ 
ac & - & - & - & - & - & - & 2 & - & - & - & - & 2 & - & 3 & - & 2\\ 
ad & - & - & - & - & - & - & - & 1 & - & - & - & - & 1 & 3 & - & 1\\ 
bc & - & - & - & - & - & - & - & - & 3 & - & - & 2 & - & - & 2 & 3\\  
bd & - & - & - & - & - & - & - & - & - & 1 & - & - & 1 & - & 2 & 3\\  
cd & - & - & - & - & - & - & - & - & - & - & 2 & - & - & 3 & 2 & 2\\  
 \end{array}
 $

 \end{center}
 \caption{Two examples of colorings $\rho_V$, $V \subseteq X$, for $X=\{a,b,c,d\}$, $c=3$ and $k=1$ or $k=2$, respectively.
 For the sake of readibilty we write $\rho_\cdot$ instead of $\rho_{\{\cdot\}}$ and we also ommit commas separating distinct elements of $X$.}\label{fig:colorings}
\end{figure}

We now prove that the situation with several colorings admits a Ramsey type result. 
For that we need to adapt the notion of ``monochromatic subset'' to our setting.
Let $Y \subseteq X$ and $m$ be a positive integer. Let $M: \binom{Y}{m} \to 2^{X\setminus Y}$ be a map which maps $m$-element subsets of $Y$ to subsets of $X\setminus Y$. To simplify a notation, we put $\widehat Z := M(Z)$.
A map $M \colon \binom{Y}{m} \to 2^{X\setminus Y}$ is called \emph{$k$-monochromatic}, if each $m$-element subset $Z$ of $Y$ is monochromatic with respect to the coloring $\rho_{Z \cup \widehat Z} : \binom{Z \cup \widehat Z}{k} \to [c]$ (that is, all $k$-element subsets of $Z$ get the same color in the coloring $\rho_{Z \cup \widehat Z}$). 

For example, assume that $X=\{a,b,c,d\}$, $Y= \{a,b,c\}$, $m=2$, and the colorings $\rho_V$, $V \subseteq X$, are defined as in Figure \ref{fig:colorings}(I). Then the trivial map $M_1$ defined as $M_1(Z) = \emptyset$ for any $Z \in \binom{Y}{2}$, is not 1-monochromatic as the set $\{b,c\}$ is not monochromatic with respect to $\rho_{bc}$, which is illustrated on Figure \ref{fig:monochromatic_example}. However, there is a 1-monochromatic map $M_2$ from $\binom{Y}{2}$ to $2^{\{d\}}$, see Figure \ref{fig:monochromatic_example} for details.

To illustrate the concept of 2-monochromatic coloring let $X=Y = \{a,b,c,d\}, m=3$ and let the colorings $\rho_V$, $V \subseteq X$, be defined as in Figure \ref{fig:colorings}(II). Then the map $M_3: \binom{Y}{2} \to 2^{\{c\}}$ defined at Figure \ref{fig:monochromatic_example} is not 2-monochromatic whereas the trivial map $M_4$ is. We refer to Figure \ref{fig:monochromatic_example} for details.

\begin{figure}
\begin{center}
$
\begin{array}{c|c|c|c}
\text{  }&Z & \widehat Z & Z \text{ monochromomatic with respect to } \rho_{Z \cup \widehat Z}\\[2pt]
\hline\hline\\[-8pt]
&\{a,b\} & \emptyset & \text{ YES: }\quad \rho_{ab}(a) = \rho_{ab}(b)\\[2pt]
M_1&\{b,c\} & \emptyset & \text{ NO: }\quad \rho_{bc}(b) \neq \rho_{bc}(c)\\[2pt]
&\{a,c\} & \emptyset & \text{ YES: }\quad \rho_{ac}(a) = \rho_{ac}(c)\\ [2pt]\cline{2-4}\hline\\[-8pt]
&\{a,b\} & \emptyset & \text{ YES: }\qquad \rho_{ab}(a) = \rho_{ab}(b)\\[2pt]
M_2&\{b,c\} & \{d\} & \text{ YES: }\ \rho_{bcd}(b) = \rho_{bcd}(c) = \rho_{bcd}(d)\\[2pt]
&\{a,c\} & \emptyset & \text{ YES: }\qquad \rho_{ac}(a) = \rho_{ac}(c)\\[2pt]
\hline\hline\\[-8pt]
&\{a,b,c\} & \emptyset & \text{ YES: }\quad \rho_{abc}(ab) = \rho_{abcd}(ac)=\rho_{abc}(bc)\\[2pt]%
M_3&\{a,b,d\} & \{c\} & \text{ NO: }\qquad\qquad \rho_{abcd}(ab) \neq \rho_{abcd}(ac)\\[2pt]%
&\{a,c,d\} & \emptyset & \text{ YES: }\quad \rho_{acd}(ac) = \rho_{acd}(ad)=\rho_{acd}(cd)\\[2pt]%
&\{b,c,d\} & \emptyset & \text{ YES: }\quad \rho_{bcd}(bc) = \rho_{bcd}(bd)=\rho_{bcd}(cd)\\[2pt]
\hline\\[-8pt]
&\{a,b,c\} & \emptyset & \text{ YES: }\quad \rho_{abc}(ab) = \rho_{abc}(ac)=\rho_{abc}(bc)\\[2pt]%
M_4&\{a,b,d\} & \emptyset & \text{ YES: }\quad \rho_{abd}(ab) = \rho_{abd}(ad)=\rho_{abd}(bd)\\[2pt]%
&\{a,c,d\} & \emptyset & \text{ YES: }\quad \rho_{acd}(ac) = \rho_{acd}(ad)=\rho_{acd}(cd)\\[2pt]%
&\{b,c,d\} & \emptyset & \text{ YES: }\quad \rho_{bcd}(bc) = \rho_{bcd}(bd)=\rho_{bcd}(cd)\\%
\end{array}
$
\end{center}

\caption{An example of four maps $M_i: Z \mapsto \widehat Z$, where $M_2$ is 1-monochromatic while $M_1$ is not, and $M_4$ is 2-monochromatic while $M_3$ is not. We use the colorings $\rho_V$ defined in Figure \ref{fig:colorings}, I and II, respectively.}\label{fig:monochromatic_example}
\end{figure}

We say that a map $M: \binom{Y}{m} \to 2^{X\setminus Y}$ is \emph{strongly injective} if $M(Z) \cap M(Z') = \emptyset$ whenever $Z \neq Z'$.  Note that there always exists a strongly injective map as we can set $M(Z) = \emptyset$ for every $Z \in \binom{Y}{m}$. 

Here is the promised Ramsey type theorem, which enables us to perform the induction step in the proof of Proposition~\ref{p:ccm}.

\begin{proposition}\label{p:ramsey_selection}
 For any positive integers $k$, $m$, $n$, $c$ there is a constant $N_k=N_k(n;m;c)$ such that the following holds. Let $X$ be a set and for every $V
\subseteq X$ let $\rho_V \colon \binom{V}{k}\to [c]$ be a coloring of the $k$-element subsets of
$V$. If $|X| \geq N_k$, then there always exist an $n$-element subset $Y \subseteq X$ and a strongly injective $k$-monochromatic map $M: \binom{Y}{m} \to 2^{X\setminus Y}$.
\end{proposition}

\begin{remark}
The need of coloring each $k$-tuple by several different colorings $\rho_V$ reflects the fact that we are going to color a cycle $z$ by the singular homology of $\gamma(z)$ inside $\conv_\F\Phi(V)$ for various different sets $V$. It may thus easily happen that if $z, z'$ are two cycles with vertices in $V \cap W$, they have the same colors in $V$ but their colors in $W$ are different. 
\end{remark}

\begin{proof}
 Let $r=R_k(m;c)$. We claim that it is enough to take \[N_k=R_r\left(n+\binom{n}{m}\cdot(r-m);\binom{r}{m}\right).\]
 Suppose that $|X|\geq N_k$ and choose an arbitrary order of the elements of $X$.

 By the choice of $r$, if $V\in\binom{X}{r}$, then there is a subset $A\subseteq V$ of size $m$ such that $\rho_V$ assigns the same color to all $k$-tuples in $A$.
 Let us introduce another coloring $\eta$ that colors each $V\in\binom{X}{r}$ by the relative position  of the first monochromatic $m$-tuple $A$ inside $V$ (with respect to the lexicographic ordering).  
 In total, there are $\binom{r}{m}$ colors. 
 
For illustration: If $V=\{a,c,d,e,f,g,h,i,\ell\}$ and $A=\{a,c,f,h\}$ we assign $V$ the ``color'' $\{1,2,5,7\}$, since the elements of $A$ are on first, second, 5th and 7th position of $V$.

Since $|X|\geq N_k$, there is a subset $U \subseteq X$ of size $n+\binom{n}{m}\cdot\left(r-m\right)$, such that all $r$-tuples in $U$ have the same color in $\eta$, say color $\Omega$.
 Consider the set $Y'=\{1,2,\ldots, n\}$.
  Let $T_1,\ldots, T_{\binom{n}{m}}$ be all the $m$-tuples in $Y'$ in some order.
 Using induction on $l=1,2,\ldots, \binom{n}{m}$, we construct disjoint sets $N(T_l)\subseteq \Q\setminus Y'$, each of size $r-m$, 
 such that each $T_l$ has position $\Omega$ within the set $T_l\cup N(T_l)$.
 In the $l$th step, we define $B_l:=Y'\cup\bigcup_{l'<l} N(T_{l'})$.
 Since rational numbers are dense, there is some set $S\subseteq \Q\setminus B_l$
 such that $T_l$ has position $\Omega$ within $T_l\cup S$. So we let $N(T_l)$ to be the set $S$ and continue in the induction, until we exhaust all the $m$-tuples from $Y'$.
 An example illustrating this concept is presented at Figure \ref{fig:ramsey_rational_numbers}.
 
 \begin{figure}
  \begin{center}
  $
  \begin{array}{c|| >{\columncolor{black!7}}c>{\columncolor{black!7}}
  ccc>{\columncolor{black!7}}cc>{\columncolor{black!7}}cccc}
  
    T_l         & \mathbf1 & \mathbf2 & 3    & 4    & \mathbf5    & 6    & \mathbf7    & 8    & 9   \\ \hline\hline
    \ 
\\[-2.3ex]
 \{1,2,3,4\} & \mathbf1 & \mathbf2 & 2.1 & 2.2 & \mathbf3 & 3.1 & \mathbf4 & 5.1 & 5.2\\
 \{1,2,3,5\} & \mathbf1 & \mathbf2 & 2.3 & 2.4 & \mathbf3 & 4.1 & \mathbf5 & 5.3 & 5.4\\
 \{1,2,4,5\} & \mathbf1 & \mathbf2 & 3.2 & 3.3 & \mathbf4 & 4.2 & \mathbf5 & 5.5 & 5.6\\
 \{1,3,4,5\} & \mathbf1 & \mathbf3 & 3.4 & 3.5 & \mathbf4 & 4.3 & \mathbf5 & 5.7 & 5.8\\
 \{2,3,4,5\} & \mathbf2 & \mathbf3 & 3.6 & 3.7 & \mathbf4 & 4.4 & \mathbf5 & 5.9 & 6.0\\
 \end{array}
 $
 \end{center}
  \caption{An example illustrating an assignment of rational numbers in the proof of Proposition \ref{p:ramsey_selection}. Here $Y' = \{1,2,3,4,5\}$, $n=5, m=4$, $r=9$ and $\Omega = \{1,2,5,7\}$. We highlight the position $\Omega$ and distribute the rational numbers in such a way that $T_l$ is on position $\Omega$ within the set $T_l \cup N(T_l)$.}
  \label{fig:ramsey_rational_numbers}
 \end{figure}
 
 \begin{figure}
     {\footnotesize $
  \begin{array}{c||c*{36}c}
  \text{elements of } U  \text{ (part)}& \cellcolor{black!7}\mathbf{a} & \cellcolor{black!7}\mathbf{c} & d & e & f & g & \cellcolor{black!7}\mathbf{h} & i & \ell & o & p & q & r & s & \cellcolor{black!7}\mathbf{t} \\ \hline &
\\[-2.3ex]
\text{elements of } T_l \cup \bigcup N(T_l) & \cellcolor{black!7}\mathbf{1} & \cellcolor{black!7}\mathbf{2} & 2.1 & 2.2 & 2.3 & 2.4 & \cellcolor{black!7}\mathbf{3} & 3.1 & 3.2 & 3.3 & 3.4 & 3.5 & 3.6 & 3.7 & \cellcolor{black!7}\mathbf{4} \\ \\[-0.7em]
  \text{elements of } U \text{ (continuation)}   & A & C & D & E & \cellcolor{black!7}\mathbf{F} & G & H & I & L & O & P & Q & R & S & T \\\hline
\\[-2.3ex]
 \text{elements of } T_l \cup \bigcup N(T_l) & 4.1 & 4.2 & 4.3 & 4.4 & \cellcolor{black!7}\mathbf{5} & 5.1 & 5.2 & 5.3 & 5.4 & 5.5 & 5.6 & 5.7 & 5.8 & 5.9 & 6.0
  \end{array}
 $ }
  
  \caption{An example showing the unique order-preserving isomorphism between $U$ of size 30 and rational numbers given in Figure \ref{fig:ramsey_rational_numbers}. \label{fig:bijection} Here $U = \{a,c,d,e,f,g,h,i,j,k,\ell,o,p,q,r,s,t, A,C,D,E,F,G,H,I,L,O,P,Q,R,S,T\}$ is an ordered set where all small letters are ordered first (lexicographically), followed by lexicographically ordered capital letters. The isomorphism also determines the set $Y$ as it is the image of the set $Y'=\{1,2,3,4,5\}$, that is $\{a,c,h,t,F\}$.}\label{fig:iso}
 \end{figure}

\begin{figure}
\begin{center}
$
 \begin{array}{ccc}
  M: \qquad \binom{Y}{m} & \to & 2^{U\setminus Y} \\[2pt]
  \qquad \qquad \{a,c,h,t\} & \mapsto & \{d,e,i,G,H\}\\[0.5pt]
  \qquad \qquad \{a,c,h,F\} & \mapsto & \{f,g,A,I,L\}\\[0.5pt]
  \qquad \qquad \{a,c,t,F\} & \mapsto & \{\ell,o,C,O,P\}\\[0.5pt]
  \qquad \qquad \{a,h,t,F\} & \mapsto & \{p,q,D,Q,R\}\\[0.5pt]
  \qquad \qquad \{c,h,t,F\} & \mapsto & \{r,s,E,S,T\}
 \end{array}
$
\end{center}
\caption{An example illustrating the construction of the map $M$ in the proof of Proposition \ref{p:ramsey_selection}. Here $m=4$ and both  $Y=\{a,c,h,t,F\}$ and $U$ are same as in Figure \ref{fig:bijection}. }\label{fig:M}
\end{figure}

 Let $\iota$ be the unique order-preserving isomorphism from $Y'\cup\bigcup_{l} N(T_l)$
 to $U$. 
 We consider the natural extension of $\iota$ to all subsets $X$ of $Y'\cup\bigcup_{l} N(T_l)$ given by  $\iota(X):=\{\iota(x)\colon x\in X\}$.
 We then set $Y=\iota(Y')$ and define the required map $M$ via $M(\iota(T_l))=\iota(N(T_l))$; see Figures \ref{fig:iso} and \ref{fig:M} describing our running example.
 By construction, the map $M: \binom{Y}{m} \to 2^{X \setminus Y}$ is strongly injective as $N(T_l)$ are pairwise disjoint.  
 It also follows from the construction that  each set $S\in \binom{Y}{m}$ is on position $\Omega$ within $S\cup M(S)$. Consequently, each $S$ is monochromatic with respect to the coloring $\rho_{S\cup M(S)}: \binom{S \cup M(S)}{k} \to [c]$ by our choice of the set $U$, that is, for each $S \in \binom{Y}{m}$ all $k$-element subsets of $S$ have the same color in $\rho_{S \cup M(S)}$. 
 Therefore, $M$ is $k$-monochromatic and the result follows.
 \end{proof}

 \subsection{The induction}\label{s:induction}
 
 \begin{proof}[Proof of Proposition \ref{p:ccm}]
We proceed by induction on $\dim K$, similarly as in \cite{hb17}. If the reader finds the current exposition too fast, we encourage him/her to consult \cite{hb17} which goes slower and shows motivation and necessity of some ideas presented here. Note however, that our current setup is  more general as it works for an arbitrary \emph{closure space} (for definition see 
 Section \ref{ss:convexvsclosed}). In the following text we repeatedly use that a simplex (as a face of an abstract simplicial complex) is given by its set of vertices.  We thus implicitly identify a simplex with its set of vertices (and vice versa).

\heading{Induction basis.}
If $K$ is $0$-dimensional with vertices $V(K)=\{v_1,\ldots, v_m\}$, we set $r_K(b)=m$. If $P=\{x_1,\ldots, x_n\}$ is a point set in $\mathbf{R}$ with $|P|\geq m$, we can take as $\Phi$ the map $\Phi(v_i) = \{x_i\}$.
It remains to define $\gamma$. We want it to ``map'' $v_i$ to $x_i$.
However, $\gamma$ should be a chain map from simplicial chains of $K$ to singular chains in $\R^d$. Therefore for each vertex $v_i$ we define $\gamma(v_i)$ as the unique map from
$\Delta_0$ to $x_i$ (here we consider $\Delta_n$ to be a geometric simplex); and extend this definition linearly to the whole $C_0(K)$. By construction, $\gamma$ is nontrivial and constrained by $(\F,\Phi)$.

\heading{Induction step.}
Let $\dim K = k \geq 1$.
The aim is to find a chain map $\gamma \colon C_*(K^{(k-1)}) \to C_*(\mathbf{R})$ and a suitable map $\Phi$ such that $\gamma$ is nontrivial, constrained by $(\F,\Phi)$ and $\gamma(\partial \sigma)$ has trivial homology inside $\conv_\F\Phi(\sigma)$ for each $k$-simplex $\sigma \in K$. Extending such $\gamma$ to the whole complex $K$ is then straightforward.

 Let $s \geq 1$ be some integer depending on $K$ which we determine later. 
To construct $\gamma$ we will define three auxiliary chain maps
\[ C_*\pth{K^{(k-1)}} \quad \xrightarrow{\makebox[2em]{$\alpha$}}
\quad C_*\pth{\skel{k-1}{(\sd K)}} \quad
\xrightarrow{\makebox[3em]{$\beta$}} \quad
C_*\pth{\skelsim{k-1}{s}} \xrightarrow{\makebox[2em]{$\gamma'$}} \quad
C_*(\mathbf{R}),\]
where $\sd K$ is the barycentric subdivision of $K$.

\heading{Definition of $\alpha$.}
We start with the easiest map, $\alpha$. 
It maps each $l$-simplex $\sigma$ from $K^{(k-1)}$ to the sum of the $l$-simplices in the barycentric subdivision of $\sigma$. 
 
 \heading{Definition of $\gamma'$.} The map $\gamma'$ is obtained from induction.
 Let the cardinality of $P$ be large enough. Since $\dim \skelsim{k-1}{s} = k-1$, by induction hypothesis, there is a nontrivial chain map $\gamma':C_*(\skelsim{k-1}{s}) \to C_*(\mathbf{R})$ and a map $\Psi_0\colon \skelsim{k-1}{s}\to 2^P$ such that $\gamma'$ is constrained by $(\F,\Psi_0)$. 
 
 For each simplex $\sigma\in \Delta_s$ we define
 
 \begin{equation}\label{eq:psi} \Psi(\sigma)=\bigcup_{\tau
  \in \skelsim{k-1}{s}, \tau \subseteq \sigma}\Psi_0(\tau).
  \end{equation}
  
  First note that $\Psi$ extends $\Psi_0$. Indeed, for $\sigma \in \skelsim{k-1}{s}$, 
  $
    \Psi(\sigma) = \bigcup_{\tau \subseteq \sigma}\Psi_0(\tau) = \Psi_0(\sigma),
   $
where the last equality follows from the fact that  $\Psi_0$ is monotone (see Remark \ref{r:monotonicity}) so the union is equal to the largest term (with respect to inclusion).
  Moreover, easy calculation shows that 
    $\Psi(A) \cap
  \Psi(B) = \Psi(A \cap B)$ for any $A, B \subseteq V(\Delta_s)$.

 \heading{Definition of $\beta$.} 
 With the help of Proposition~\ref{p:ramsey_selection} it is now easy to find the map $\beta$. 
 Indeed, for each simplex $\tau\in\Delta_s$, let $\rho_\tau$ be the coloring that assigns to each $k$-simplex $\sigma\subseteq\tau$ the singular homology class of $\gamma'(\partial\sigma)$ inside $\conv_{\mathcal F}(\Psi(\tau))$.
 Let $m$ be the number of vertices of $\sd \Delta_k$, $n$ the number of vertices of $\sd K$
 and $c$ the maximal number of elements in $\widetilde H_k(\bigcap \mathcal G;\Z_2)$, where $\mathcal G\subseteq\F$. Clearly $c\leq 2^b$. 
 
 Thus if $s\geq N_{k+1}(n;m;c)$ from Proposition~\ref{p:ramsey_selection}, the following holds:
 
 \begin{enumerate}
 \item There is an inclusion $j$ of
 $\left(\sd K\right)^{(k-1)}$ to a simplex $Y\subseteq\Delta_s$. We let $\varphi\colon K\to 2^{V(\Delta_s)}$ be the map that to each $\sigma\in K$ assigns the set $j(V(\sd \sigma))$. In particular, $\varphi(\sigma) \subseteq Y$ for each $\sigma \in K$.\\
 \item For each $k$-simplex $\mu$ in $K$ there is a simplex $\widehat\mu := M(\sd \mu)$ in $\Delta_s$, where $M: \binom{Y}{m} \to 2^{V(\Delta_s)\setminus Y}$ is the strongly injective $(k+1)$-monochromatic map given by Proposition \ref{p:ramsey_selection}. The following three properties of $\widehat \mu$ are satisfied:\\ 
 
 \begin{enumerate}[label=(\roman*)]\label{e:prop}
    \item \emph{((k+1)-monochromaticity)} For all $k$-simplices $\tau$ in $\sd \mu$,   the singular homology class of $\gamma'(j(\partial\tau))$ inside $\conv_{\mathcal F}\Psi(\widehat\mu\cup \varphi(\mu))$ is the same,\\
  \item each $\widehat\mu$ is disjoint from $Y$, \label{e:disjoint_MY}\\
  \item \emph{(strong injectivity)} all the simplices $\widehat\mu$ are mutually disjoint.\label{e:disjoint_M}\\
 \end{enumerate}
 \end{enumerate}
 
 We define $\widehat\mu:=\emptyset$ for $\mu \in K$ a simplex of dimension at most $k-1$. We set $\Phi(\mu):=\Psi(\widehat\mu\cup \varphi(\mu))$.
 Note that for a simplex $\sigma\in K^{(k-1)}$, $\Phi(\sigma)$ reduces to $\Psi(\varphi(\sigma))$.
 
Let $\beta$ be the chain map induced by $j$.
 Observe that $\Phi$ satisfies $\Phi(\emptyset)=\emptyset$ and $\Phi(A\cap B)=\Phi(A)\cap\Phi(B)$, $A,B\in K$. Indeed, the first claim is obvious and for the second one let $\mu, \tau$ be distinct simplices in $K$:
 \begin{align*}
   \Phi(\mu) \cap \Phi(\tau) &= \Psi\left(\widehat\mu \cup \varphi(\mu)\right) \cap \Psi\left(\widehat\tau \cup \varphi(\tau)\right) = \Psi\left([\widehat\mu \cup \varphi(\mu)] \cap [\widehat\tau \cup \varphi(\tau)]\right) \\
   &= \Psi(\varphi(\mu) \cap \varphi(\tau)),
 \end{align*}
 where the second equality express the fact that $\Psi$ respects intersections and the last equality uses \ref{e:prop}\ref{e:disjoint_MY}, \ref{e:prop}\ref{e:disjoint_M} and the fact that $\varphi(\mu) \subseteq Y$ for every $\mu \in K$. Then
 \[
  \Phi(\mu) \cap \Phi(\tau) = \Psi(\varphi(\mu) \cap \varphi(\tau))=  \Psi (\varphi(\mu \cap \tau)) = \Phi(\mu \cap \tau)
 \]
since $\varphi$ obviously respects intersections and $\dim (\mu \cap \tau) \leq k-1$.
 
 \medskip
 We define $\gamma$ on $K^{(k-1)}$ as the composition $\gamma'\circ \beta\circ\alpha$.
 Then, by the definition, $\gamma$ is a nontrivial chain map constrained by $(\F, \Phi)$.
 
 As the last step, we extend $\gamma$ to the whole complex $K$.
 If $\sigma$ is a $k$-simplex of $K$,
 all the $k$-simplices $\zeta$ in $\sd\sigma$ have the same value of $\gamma'\beta(\partial\zeta)$ inside $\conv_\F\Phi(\sigma)$. Since there is an even number of them and we work with $\Z_2$-coefficients, $\gamma(\partial\sigma)$ has trivial homology inside $\conv_{\F}\Phi(\sigma)$. So for each such $\sigma$ we may pick some $\gamma_\sigma\in C_k\left(\conv_{\F}\Phi(\sigma);\Z_2\right)$ such that $\partial\gamma_\sigma=\gamma(\partial\sigma)$ and extend $\gamma$ by setting $\gamma(\sigma):=\gamma_\sigma$.
 Then, by definition, $\gamma$ is a non-trivial chain map from $C_*(K;\Z_2)$ to $C_*(\mathbf{R};\Z_2)$ constrained by $(\F,\Phi)$ and hence by $(\F,P)$.
 \bigskip
 
 It remains to show that if $\dim K\leq 1$, we can take $\gamma$ that is induced by a continuous map $f \colon |K| \to \mathbf{R}$.
 If $\dim K=0$, we map each point to a point, so the statement is obviously true.
 
 If $\dim K=1$, we inspect the composition $\gamma=\gamma'\circ \beta\circ \alpha$. 
 It maps points of $K$ to points in $\mathbf{R}$
 in such a way that the homology class of $\gamma(\partial \tau)$ inside $\conv_{\F}(\Psi(\tau))$ is trivial for each edge $\tau$ of $K$. But this means that the endpoints of $\tau$ get mapped to points in the same path-connected component of $\conv_{\F}(\Psi(\tau))$ and can be connected by an actual path.
\end{proof}

\section{A fractional Helly theorem on surfaces}\label{s:pq_surface}
In order to prove Theorem \ref{t:frachelly_surfaces}, we need to bring the constant $m$ from Theorem \ref{t:non-optimal-frachelly} (applied to a surface $S$) down to three for $b=0$ and to $2b+4$ for $b \geq 1$, respectively.
The presented  method is based on the recent result of Kalai and the author \cite{planar_sets} and allows us to significantly decrease $m$ to a small value as soon as we have a finite upper bound.

We need few definitions first. Let $\A=\{A_1,\ldots,A_n\}$ be subsets of a surface $S$. Set $A_I = \bigcap_{i \in I}A_i$ and let $N(\A) = \{I \in [n] \colon A_I \neq \emptyset\}$ be the nerve of $\A$. We put $f_k(\A) = |\{I \in N(\A) \colon |I| = k+1\}|$.  In words, $f_k$ counts the number of intersecting $(k+1)$-tuples from $\A$. 

The main tool for the bootstrapping is the following proposition.
\begin{proposition}\label{p:bootstraping}
 Let $b \geq 0$ and $k \geq 2$ be integers satisfying that for  $b=0$, $k \geq 2$, and for $b \geq 1$, $k > 2b + 2$, respectively. Let $S$ be a surface.
 Then for every $\alpha\in (0,1)$ there exists $\alpha'=\alpha'(\alpha,b,k,S)> 0$ such that for any sufficiently large family $\A$ of $n$ open sets in $S$
 with $HC_1(\A) \leq b$ and every integer $r\geq 1$, the following holds:
\begin{equation}
 f_k(\A) \geq \alpha\binom{n}{k+1} \quad \Rightarrow \quad f_{k+r}(\A) \geq \alpha'\binom{n}{k+r+1}. \label{eq:bootstrap}
\end{equation}
\end{proposition}
\begin{proof}[Proof of Theorem \ref{t:frachelly_surfaces} assuming Proposition \ref{p:bootstraping}]
Let $b \geq 0$ and let $k_0 = k_0(b)$ be an integer depending on $b$. Namely, we set $k_0(0)=3$ and $k_0(b) = 2b+4$ for $b \geq 1$. Let $m$ be the value from the (non-optimal) fractional Helly theorem (Theorem \ref{t:non-optimal-frachelly}). We assume that $m \geq k_0 + 1$, otherwise we  just take  a minimal integer which satisfies both conditions.
By Proposition \ref{p:bootstraping} ($k=k_0-1, r = m-k_0$) we get that if at least an $\alpha$-fraction of all $k_0$-tuples intersect, then also some $\alpha'$-fraction of all $m$-tuples intersect. 
Theorem \ref{t:non-optimal-frachelly} now implies that there is a point common to some $\beta$-fraction of all the sets. This finishes the proof.
\end{proof}

It remains to prove  Proposition \ref{p:bootstraping}. As mentioned, this proof heavily relies on \cite[Theorem 2.3]{planar_sets}, which can be reformulated, for open sets and in terms of bounded homological complexity, as follows:
\begin{theorem}[{\cite{planar_sets}}]\label{t:GZ}
 Let $S$ be a surface, $b \geq 0$ an integer and let $k = k(b)$ be an integer depending on $b$, namely $k(0) \geq 2$ and $k(b)  > 2b+2$ for $b \geq 1$. 
 Then there exist constants $c_1 > 0, c_2 \geq 0$ depending only on $k,b$ and $S$ such that the following holds.
 Let  $\A$ be a finite family of open sets in $S$ with $HC_1(\A) \leq b$. Then
 \[  f_{k}(\A) > c_1f_{k-1}(\A) + c_2 \quad \Rightarrow \quad f_{k+1}(\A) > 0.\]
\end{theorem}

Since a nerve of any family is a simplicial complex, hence a hypergraph, we will recall some useful facts about hypergraphs.
\heading{Hypergraphs.}
A hypergraph is  \emph{$\ell$-uniform} if all its edges have size $\ell$.
A hypergraph is \emph{$\ell$-partite}, if its vertex set $V$ can be partitioned into $\ell$ subsets $V_1, \ldots, V_\ell$,
called \emph{classes}, so that each edge contains at most one point from each $V_i$.
Let $K^\ell(t)$ denote the complete $\ell$-partite $\ell$-uniform hypergraph with $t$ vertices in each of its $\ell$ vertex classes.

We need the following theorem of Erd\H{o}s and Simonovits \cite{Erdos-Simonovits} about super-saturated hypergraphs (see also \cite[Chapter 9.2]{matousek_lectures}):
\begin{theorem}[\cite{Erdos-Simonovits}]\label{t:erdos_simonovits}
 For any positive integers $\ell$ and $t$ and any $\varepsilon > 0$ there exists $\delta > 0$ with the following property:
Let $H$ be an $\ell$-uniform hypergraph on $n$ vertices and with at least $\varepsilon \binom{n}{\ell}$ edges. Then $H$
contains at least $ \delta n^{\ell t}$ copies  of $K^\ell(t)$.
\end{theorem}
\begin{proof}[Proof of Proposition \ref{p:bootstraping}]
Let $\A = \{A_1,\ldots,A_n\} $ be a family of sets in $S$ satisfying the assumptions. Since $f_{k-1}(\A) \leq \binom{n}{k}$, Theorem \ref{t:GZ} gives
\begin{eqnarray}\label{e:quant}
   f_k(\A) > c_1 \binom{n}{k} + c_2 \quad \Rightarrow \quad f_{k+1}(\A) > 0
\end{eqnarray}
for some constants $c_1> 0, c_2 \geq 0$.

Note that it is enough to show (\ref{eq:bootstrap}) for $r=1$ as the rest can be obtained by induction.

 Let $H$ be a $(k+1)$-uniform hypergraph whose vertices and edges correspond to the vertices and $k$-simplices
 of the nerve $N$ of $\A$. By assumption $H$ has at least $\alpha \binom{n}{k+1}$ edges. Set
  \[
  t := \left\lceil \frac{c_1(k+1)^k}{k!}+c_2\right\rceil
 \]

By Theorem \ref{t:erdos_simonovits} ($\varepsilon=\alpha, \ell = k+1$),
 there is at least $\delta n^{(k+1)t}$ copies of $K^{k+1}(t)$ in $H$.

 Since $K^{k+1}(t)$ has $(k+1)t$ vertices and $t^{k+1}$ edges, it follows by (\ref{e:quant}) that for every copy of $K^{k+1}(t)$ in $H$
 there is an intersecting subfamily of size $k+2$ among the corresponding members of $\mathcal A$. 
 Indeed, the implication (\ref{e:quant}) translates into checking that for $k \geq 2$, $t \geq 1$,
 \begin{eqnarray*}
  t^{k+1} &>& c_1\binom{(k+1)t}{k} + c_2.
 \end{eqnarray*}
 
On the other hand, each such intersecting $(k+2)$-tuple is contained in at most $n^{(k+1)t-(k+2)}$ distinct copies of $K^{k+1}(t)$
 (we count the number of choices for the vertices not belonging to the considered $(k+2)$-tuple), and the result follows,
 i.e. $f_{k+1}(\A) \geq \delta n^{k+2} \geq \alpha' \binom{n}{k+2}$.
\end{proof}

\section{Discussion}\label{s:discussion}
 
\subsection{Qualitative sharpness of Theorem \ref{t:boundedRadon}.}\label{ss:sharpness}
 \begin{example}\label{ex:sharpness}The bound on the Radon number provided by Theorem~\ref{t:boundedRadon} is extremly large as it is obtained by an iterative application of Ramsey theorem. Nevertheless, the result is optimal in the following sense:  all $\widetilde \beta_i$, $0 \leq i < \lceil \frac{d}{2}\rceil$ need to be bounded. We show that we can obtain arbitrarily large Radon number as soon as at least one $\widetilde \beta_i$ is unbounded. More precisely, we show that for every $0\leq k < \lceil \frac{d}{2}\rceil$ and all $n\in \N$ we can find a family $\F$, for which $\widetilde{\beta_i} \left(\bigcap \G\right)=0$ for all $i\neq k$ and all $\emptyset\neq\G\subseteq\F$, yet the Radon number of $\F$ is at least $n$. This shows that unless we bound all the first $\lceil d/2\rceil$ reduced Betti numbers, we cannot obtain a universal bound on $r(\F)$ that would be valid for all families $\F$  in $\mathbb R^d$. 
 Indeed, fix $0 \leq k < \lceil \frac{d}{2}\rceil$ and let $n$ be arbitrarily large.  
 Let $K$ be a realization of the $k$-skeleton of the $(n-1)$-dimensional simplex in $\R^d$ (see \cite[Section 1.6]{Matousek_borsuk_ulam}); more precisely, let $V(K) = \{v_1, \ldots, v_n\}$ be a set of $n$ points in general position in $\R^d$, then each subset $A \subseteq V(K)$ of size at most $k+1$ spans a geometric simplex $\conv A$. Here $\conv(\cdot)$ stands for the standard convex hull in $\R^d$.
 For each $i \in [n]$, let $F_i$ be the union of all simplices in $K$ not containing the vertex $v_i$, put $\F= \{F_i\colon i \in [n]\}$. We show that $r(\F) > n$.  More precisely, we show that for any two disjoint subsets $B_1, B_2 \subseteq V(K)$ we have $\conv_\F(B_1) \cap \conv_\F(B_2) = \emptyset$. Observe that for $C \subsetneq [n]$, $\conv_\F \{ v_i\colon i \in C \}$ is the induced subcomplex of $K$ spanned by the vertices $\{v_i \colon v_i \in C\}$ as $\conv_\F \{ v_i\colon i \in C \} = \bigcap_{i \in [n]\setminus C} F_i$. This implies that $r(\F) > n$ as in any partition of $V(K)$ the convex hulls of the parts cannot intersect.
 Given a subfamily $\G \subseteq \F$ of cardinality $m$ ($m \leq n$), $\bigcap \G$ is the $k$-skeleton of a simplex on $n-m$ vertices, and hence have all reduced Betti numbers equal to zero except in dimension $k$, where the reduced Betti number equals $\binom{n-m}{k+1}$. 
   
 This construction has been used before (\cite[Example 3]{hb17}) to show that all $\widetilde \beta_i$, $0 \leq i < \lceil \frac{d}{2}\rceil$, have to be bounded in order to have a bounded Helly number. 
 \end{example}

\subsection{$\mathbb Z_2$-coefficients and manifolds.}\label{ss:Z2-coefficients}
 A curious reader might wonder why do we compute over $\mathbb Z_2$. Could we possibly get a better bound using another ring like $\Z$ or $\Z_p$?

First reason is that so far the homological van Kampen-Flores theorem (Theorem \ref{t:homological_van_Kampen}) is proven only for $\Z_2$. Using ideas from \"Ozaydin's paper~\cite{ozaydin}, it would be possible to prove an analogous result over $\Z_p$, but then we would need to consider a $(1-1/p)d$-dimensional complex in Proposition~\ref{p:gen}.
Since raising the target dimension by one requires twofold application of Ramsey theorem, 
such procedure is unlikely to improve the resulting bound on the Radon number.
Moreover, one would need to restrict the homological complexity up to that dimension.
In other words, it would not be enough to bound the complexity only up to mid-level. Nevertheless, there still might be rare situations in which such approach results in slightly better bounds on Tverberg numbers
than the ones that follow from combination of~\cite[(6)]{jamison1981} or \cite[Thm 1.1]{domotor}, respectively, and our bound on the Radon number.

The second reason is that the choice of $\Z_2$ enables us to completely separate the Ramsey argument (that is, no homology is involved in Proposition \ref{p:ramsey_selection}). One big advantage of having the topological and combinatorial parts of the proof independent  is that once there is an extension of the homological van Kampen-Flores theorem to manifolds, we readily get most of the results in this paper also for families on manifolds.

\begin{problem}
 Is there a homological van Kampen-Flores theorem for manifolds?
\end{problem}

It might seem that the need for a homological van Kampen-Flores for manifolds can be overcome by the approach mentioned in Section~\ref{s:embeddability}, as we can often embed any reasonably nice $d$-manifold into $\R^{2d}$. For example, combining the embedding with Theorem \ref{t:non-optimal-frachelly} applied to $\mathcal F^\cap$ (see Section \ref{s:direct_conseq}, \textsection\ $(p,q)$-theorems), we get a $(p,q)$-theorem for manifolds. However, $p\geq m$ has to be sufficiently large as the bound $m$ on the fractional Helly number is large. Moreover, we need to bound homological complexity of $\F$ up to level $d$.
This brings a question, can we improve the bound on the Radon number?

\begin{problem}\label{p:Radon_bounds}
 Provide lower and upper bounds on the Radon number from Theorem \ref{t:boundedRadon}.
\end{problem}

This would be helpful not only from the perspective of the manifolds, but also for the fractional Helly theorems, as the dependency on the Radon number is hidden in the bound on $\beta$ in Theorem \ref{t:frachelly_surfaces}. Indeed, Radon number tells us how many steps in the bootstraping described in Section~\ref{s:pq_surface} we need to perform.

What our method gives, regarding Problem \ref{p:Radon_bounds}, are the following partial results.

\begin{theorem}
 For $\F \subseteq \R^d$ with $HC_{\lceil d/2\rceil}(\F) = 0$ 
 the Radon number $r(\F)$ is at most $d+3$.
 In case that $HC_d(\F) = 0$, we have that $r(\F) \leq d+2$. Moreover, the first bound is sharp for $d\geq 2$, whereas the second bound is sharp for $d\geq 0$.
\end{theorem}

\begin{proof}
 In both cases we proceed exactly as in the proof of Theorem \ref{t:boundedRadon},
 the only difference is that we do not need to use Ramsey theorem and barycentric subdivisions. 
 More precisely we identify the points of the constructed complex with the original point set $P$ and set $\Psi(\sigma)=\sigma$ for every face of the constructed complex.
 In the induction procedure, the boundary of each simplex $\partial\tau$ lies in the 
 set $\conv_\F(\sigma)$, which is homologically trivial and hence can be ``filled'' by some cycle.
 Thus in the first case we manage to construct a non-trivial chain map from $\skelsim{\lceil d/2\rceil}{d+2}$ to $\R^d$ constrained by $(\F, P)$. The rest of the proof is the same, we use Theorem \ref{t:homological_van_Kampen} in combination with Proposition \ref{p:gen}. 
 
 In the second case we construct a non-trivial chain map from the boundary of $(d+1)$-simplex (which is $\skelsim{d}{d+1}$) to $\R^d$ constrained by $(\F, P)$. By \cite[Lemma 15]{hb17},  this simplicial complex admits no homological almost-embedding into $\R^d$, so Proposition \ref{p:gen} yields the result.
 
 It remains to show the sharpness. The case $HC_d(\F) = 0$ is simple, take as $\F$ the family of all $(d-1)$-faces of the $d$-simplex $\Delta_d$. $\F$ clearly has a geometric realization in $\R^d$, $HC_d(\F) = 0$ as all sets are convex, and $r(\F)  > d+1$ as $\conv_\F G = \conv G$ for any $G \subsetneq V(\Delta_d)$. Note that this is the classical example that Radon theorem of standard convex sets is tight. 
 
 Now we construct a family $\F$ with $HC_{\lceil d/2\rceil}(\F) = 0$ and $r(\F) > d+2$. The construction and its analysis is  similar to the one in Example \ref{ex:sharpness}.
 Let $K$ be the $(d-1)$-skeleton of the $(d+1)$-simplex $\Delta_{d+1}$ on the vertex set $V(K) = \{v_1, \ldots, v_{d+2}\}$. 
 We show that $K$ can be realized in $\R^d$. We start with a (geometric) $d$-simplex $\Delta_d$ in  $\R^d$, say with vertices $V(K) \setminus \{v_{d+2}\}$. Now we consider the barycenter $v_{d+2}$ of $\Delta_d$ and take the $(d-1)$-skeleton of \emph{stellar subdivision} $(\partial \Delta_d * v_{d+2})$. What we get is a realization of $K$ in $\R^d$.
 Let $\F = \{F_i \colon i \in [d+2]\}$, where $F_i$ is a union of all faces of $K$ not containing the vertex $v_i$. In other words, $F_i$'s are the boundaries of $d$-faces of $\Delta_{d+1}$.  Note that  $r(\F) > d+2$ by exactly the same reasons as before.
 Given a subfamily $\G \subseteq \F$ of cardinality $m \geq 2$, $\bigcap \G$ is a $(d+1-m)$-simplex, hence it has all reduced Betti numbers trivial. If $|\G| = 1$, i.e. $\G =\{F_i\}$, we have that all $\tilde\beta_j(F_i) = 0$ for $0\leq j \leq d-2$ and $\tilde\beta_{d-1}(F_i) = 1$ as $F_i$ is the boundary of $d$-simplex. As $d > \lceil d/2\rceil$ for $d \geq 2$, we conclude that $HC_{\lceil d/2\rceil}(\F) = 0$.  
\end{proof}

\subsection{Relation to convexity spaces}\label{ss:convexvsclosed}
The proofs of Theorems \ref{t:non-optimal-frachelly}, \ref{t:fract_helly}, and \ref{t:pq_surface} rely on a fractional Helly theorem \cite[Thm 1.1]{boundedRadon_fractHelly}, which is formulated for \emph{convexity spaces}. The same concerns the implications that bounded Radon number implies bounded Tverberg number \cite{jamison1981, domotor}
as well as weak colorful Helly theorem \cite[Thm 2.2]{boundedRadon_fractHelly}. However, for a family $\mathcal F$ of subsets of $X$, a family $\mathcal C'_\mathcal F:= \{\conv_\mathcal F(S): S \subseteq X\} \cup \{\emptyset\}$ only gives a \emph{closure space} (defined below). Nevertheless, we show on examples \cite[Thm 1.1]{boundedRadon_fractHelly}, \cite[(6)]{jamison1981} that the relevant theorems  also hold in the more general setting of closure spaces. 

We start with definitions. A \emph{convexity space} is a pair $(X, \mathcal C)$, where $X$ is any set of points and $\mathcal C$ is any family over $X$ that (i) contains $\emptyset, X$, (ii) is closed under (arbitrary) intersection and (iii) is closed under (arbitrary) union of nested sets. A pair $(X,\mathcal C')$ satisfying only (i) and (ii) is called a \emph{closure space} and the elements of $\mathcal C'$ are called \emph{closed} sets \cite{vel93}. 
The definition of Radon number naturally applies to a convexity/closure space $(X,\mathcal C$), being defined as $r(\mathcal C)$. Note that the Radon number of a family $\mathcal F$  of subsets of $X$ equals to the Radon number of the corresponding closure space $(X, \mathcal C'_\mathcal F)$, where $\mathcal C'_\mathcal F$ is defined above.

When the underlying space $X$ is clear, we write only $\mathcal C$ or $\mathcal C'$ to denote a convexity or a closure space, respectively.  We note that if a family $\mathcal C'$ of closed sets is finite, then $\mathcal C'$ is in fact a convexity space, as the condition (iii) is trivially satisfied.
Furthermore, it is not difficult to see that in each of the results (\cite[Thm 1.1]{boundedRadon_fractHelly} and \cite{jamison1981, domotor}), we only deal with finite families $\mathcal C'$. 

Let us illustrate it first on the fractional Helly theorem (\cite[Thm 1.1]{boundedRadon_fractHelly}:
Suppose that we know that for each convexity space $(X,\mathcal C)$
whose Radon number is at most $r$, the fractional Helly number is at most~$D$.
We recall that it means that for every $\alpha \in (0,1)$ there is $\beta>0$ such that the following holds: for every finite subfamily $\mathcal G\subseteq \mathcal C$, 
if an $\alpha$-fraction of all the $D$-tuples of $\mathcal G$ intersect, then there is a point contained in at least a $\beta$-fraction of all the sets in $\mathcal G$.
So let $(X,\C')$ be a closure space with Radon number $r'$
and let $\G$ be a finite subfamily of $\C'$.
Suppose that an $\alpha$-fraction of all the $D$-tuples of $\G$ intersect.
Then, since $\mathcal G \subseteq \mathcal C'$ and
$\conv_\mathcal G S \supseteq \conv_{\mathcal C'} S$, we have $r(\mathcal G)\leq r(\mathcal C') = r'$ (in other words, Radon number is monotone). 
Since the closure space generated by $\G$ contains only finitely many closed sets, it is a convexity space, thus by \cite[Thm 1.1]{boundedRadon_fractHelly}
there is a point common to $\beta$-fraction of all the sets in $\G$.
Since this is true for any finite family $\G$, this means that the fractional Helly number of the whole system $\mathcal C'$ is at most~$D$.

In order to show that bounded Radon number for a closure space implies bounded Tverberg numbers, the argument is similar. However, it is not apparently clear how to reduce the statement to a convexity space with finitely many closed sets, so let us present the argument how this is done: Suppose that we know that any convexity space whose Radon number is at most~$r$ has $k$th  Tverberg number $D$. Then if $(X,\mathcal C)$ is a closure space and $P$ is a set of $D$ points in $X$, we consider the closure space $(X, \mathcal D)$, where $\mathcal D := \{\conv_{\mathcal C}S: S \subseteq P\} $. By monotonicity, the Radon number $r(\mathcal D)$ is at most~$r$. Moreover, it is a finite closure space, thus a convexity space. Therefore, by our assumption, there are $k$ disjoint subsets $A_1,\ldots A_k$ of $P$ with $\bigcap_{i=1}^k\conv_{\mathcal D}A_i\neq\emptyset$. But, by our choice of $\mathcal D$, $\conv_{\mathcal D}A_i=\conv_{\mathcal C}A_i$. 
Since such sets $A_i$ can be found for any set $P \subseteq X$ of $D$ points,
the original closure space has $k$th Tverberg number at most $D$.

\section*{Funding} 
The work was supported by the GA\v CR grant no. 22-19073S and also by the Charles University projects UNCE/SCI/022 and PRIMUS/21/SCI/014.

 \section*{Acknowledgements} 
 I am very grateful to Pavel Pat\'ak for numerous discussions, valuable suggestions and proofreadings. Many thanks to Xavier Goaoc for his feedback as well as to  an anonymous referee for useful comments and recommendations, which were helpful in improving the overall presentation. A preliminary version of this work appeared in the proceedings of the 36th Symposium on Computational Geometry \cite{socg}.
 
\bibliographystyle{alpha}

\end{document}